\documentclass{article}
\usepackage{amsmath}
\usepackage{amsthm}
\usepackage{epsfig}
\newtheorem{theorem}{Theorem}
\newtheorem{lemma}{Lemma}[section]
\newtheorem{proposition}{Proposition}[section]
\newtheorem{corollary}{Corollary}[section]
\theoremstyle{definition}
\newtheorem{define}{Definition}

\title{Strongly Regular Graphs With The $7$-Vertex Condition}
\author{Sven Reichard\thanks{Partially supported by the School of
    Mathematics and Statistics at the University of Western
    Australia}\\Institut f\"ur Algebra\\Technische Universit\"at Dresden} 
\begin{document}
\maketitle
\begin{abstract}
The $t$-vertex condition, for an integer $t\ge 2$, was introduced by
Hestenes and Higman in 1971, providing a combinatorial invariant
defined on edges and non-edges of a graph. Finite rank 3 graphs
satisfy the condition for all values of $t$. Moreover, a long-standing
conjecture of M.~Klin asserts the existence of an integer $t_0$ such
that a graph satisfies the $t_0$-vertex condition if and only if it is
a rank 3 graph.

We construct the first infinite family of non-rank 3 
strongly regular graphs satisfying
the $7$-vertex condition. This implies that the Klin parameter $t_0$
is at least 8. The examples are the point graphs of a certain family
of generalised quadrangles.
\end{abstract}

\section{Introduction}

Strongly regular graphs (see Section~\ref{sec:srg}) occur naturally as
rank 3 representations of 
finite permutation groups, and there are many other examples. 
Whereas in principle all finite rank 3 graphs are known (see, e.g.,
\cite{KaL82}, \cite{LiS86}), the more general problem of 
classifying  the finite strongly regular graphs appears completely
hopeless. Hence, it is natural to consider properties of graphs which
are satisfied by the rank 3 graphs and also some other, but not all,
strongly regular graphs.

Hestenes and Higman \cite{HesH71}
introduced a family of
 such properties 
of this type that are called the $t$-vertex condition
for integers $t\ge 2$. These conditions are described in detail in 
Section~\ref{sec:t-vertex}. For a graph $\Gamma$ and
integer $t\ge 3$ the $t$-vertex condition leads to
 a combinatorial invariant on pairs of distinct vertices of
$\Gamma$; if 
this invariant is constant on the edges, we say that $\Gamma$
satisfies the $t$-vertex condition on edges, and similarly for
non-edges. 

The $2$-vertex condition is equivalent to  regularity of graphs, while
the $3$-vertex 
condition is equivalent to strong regularity.
Also for any $t\ge 3$, if a graph satisfies the $t$-vertex condition
then it also satisfies the $(t-1)$-vertex condition (see
Section~\ref{sec:t-vertex}). 
 In a rank 3 graph, all
edges and all non-edges are equivalent under the automorphism group,
and hence cannot be distinguished combinatorially, therefore these
graphs satisfy the $t$-vertex condition for any $t$.

There is a long-standing conjecture by M.~Klin \cite{FarKliMuz94b}
that there exists  a
number $t_0$ such that the only graphs satisfying the $t_0$-vertex condition
are the rank 3 graphs. Since for any fixed $t$ the
$t$-vertex condition can be checked in polynomial time (see Theorem~\ref{thm:polynomial}), a proof of
Klin's conjecture would have the remarkable consequence that rank 3
graphs can be recognised combinatorially in polynomial time.

Hence it is interesting to consider graphs which are not rank 3
graphs, but which satisfy the $t$-vertex condition for large $t$. 

\begin{enumerate}
\item
For $t=2$ it is easy to find regular graphs which are not strongly
regular, and hence not rank 3 graphs; a small  example being the cycle
$C_6$. 

\item
The smallest strongly regular graph ($t=3$) which is not a rank 3
graph is the well-known Shrikhande graph on 16 vertices, which can be
constructed from a Latin square of order $4$.

\item 
In 1971, Higman \cite{Hig71} gave the first examples for
$t=4$. He showed 
that the point graphs of generalised quadrangles (see
Section~\ref{sec:gq}) satisfy the 4-vertex condition. The smallest
example of a generalised quadrangle which does not admit a rank 3 group
is the unique $GQ(5,3)$ on 96 points.

\item 
In the late 1980's Ivanov \cite{Iva89} constructed the first graph with the 5-vertex
condition that was not  a rank $3$ graph. It has the parameters of a Latin
square graph $L_8(16)$ on $256$ vertices. Its first and second
sub-constituents 
(subgraphs induced by the neighbours and non-neighbours, respectively,
of a given 
vertex) of orders $120$ and $135$ are strongly regular and satisfy the
4-vertex condition. The graph on 135 vertices was the first example
known to satisfy the $4$-vertex condition and have
 an intransitive automorphism group.

\item 
Later,  Brouwer, Ivanov,  and Klin \cite{BroIvaKli89} generalised the
construction above,
obtaining an infinite series of graphs of order $4^k$
with the 4-vertex condition. These graphs in fact satisfy the
5-vertex condition as was shown by the author in \cite{Rei00}.
\end{enumerate}
For more historical details we refer to Section \ref{sec:history}.

Here, we will once more look at the point graphs of generalised
quadrangles. We show the following:
\begin{theorem}
\label{thm:5vc-gq}
The point graph of a generalised quadrangle satisfies the $5$-vertex
condition.
\end{theorem}

\begin{theorem}
\label{thm:gqss2}
For an integer $s$, the point graph of a generalised quadrangle of
order $(s,s^2)$ 
satisfies the $7$-vertex condition.
\end{theorem}

 Both results are best possible in the
sense that there is a generalised quadrangle of order $(5,3)$ which
does not satisfy 
the 6-vertex condition, and there is a generalised quadrangle of order
$(5,25)$ which does not satisfy the 8-vertex condition.

This provides us with an infinite family of graphs with
intransitive automorphism groups which satisfy the 7-vertex condition. 

\begin{corollary}
The constant $t_0$ in Klin's Conjecture is at least 8.
\end{corollary}

The author acknowledges a set of notes by A.~Pasini \cite{Pas91},
communicated by 
M.~Klin, which give the proof of the 5-vertex condition for
generalised quadrangles of order $(s,s^2)$. Thanks to M.~Klin for
attracting my attention  
to this problem, and for countless proposed improvements. He also
contributed essentially to the discussion in
Section~\ref{sec:history}.  Thanks also
to C.~Praeger, A.~Niemeyer and C.~Pech for helpful discussions, as
well as to
A.~Woldar for helping to polish the text.

\section{Strongly regular graphs}
\label{sec:srg}

All graphs considered in this text are finite and simple, i.e.,
undirected and without loops or multiple edges. Thus, a graph $\Gamma$
is a finite set $V=V(\Gamma)$ of vertices together with a binary
symmetric and 
anti-reflexive relation referred to as {\em adjacency}. We call $v=|V|$
the order of $\Gamma$.

If $x$ and $y$ are distinct vertices of $\Gamma$, we write $x\sim y$
if they  
are adjacent, and say that $y$ is a {\em neighbour} of $x$. Otherwise, we
call $y$ a {\em non-neighbour} of $x$, and write $x\not\sim y$.

Let $\Gamma(x)=\{y\in V|x\sim y\}$ denote the set of neighbours of $x$
in $\Gamma$. The {\em valency} of $x$ is defined as
$val(x)=|\Gamma(x)|$.
\begin{define}[Regular graph]
A graph $\Gamma$ is {\em regular} if there is a number $k$ such that
$val(x)=k$ for all vertices $x\in V$. In this case, $k$ is called the
{\em valency} of $\Gamma$.
\end{define}

\begin{define}[Strongly regular graph]
A graph $\Gamma$ is {\em strongly regular} if there are numbers $k$,
$\lambda$, $\mu$ such that
\begin{itemize}
\item $\Gamma$ is regular of valency $k$;
\item any two adjacent vertices of $\Gamma$ have exactly $\lambda$
  common neighbours, i.e., $|\Gamma(x)\cap\Gamma(y)|=\lambda$ whenever
  $x\sim y$.
\item any two distinct, non-adjacent vertices of $\Gamma$ have exactly
  $\mu$ common neighbours, i.e., $|\Gamma(x)\cap\Gamma(y)|=\mu$ whenever
  $x\not\sim y$.
\end{itemize}
In this case, the numbers $(v,k,\lambda,\mu)$ are called the {\em
  parameters} of $\Gamma$, where $v=|V|$ is its order.
\end{define}

We refer to Section \ref{sec:history} for a brief discussion of
numerical restrictions on putative parameters of $\Gamma$.

\section{Isoregular graphs}

There are several ways to generalise the concept of strong
regularity. One of them is the $t$-vertex condition, which we will
discuss in Section \ref{sec:t-vertex}. Another one is isoregularity.

Let $\Gamma=(V,E)$ be a graph, and let $S\subset V$ be a set of
vertices. We define the valency of $S$ as
\[
val(S) = \left|\bigcap_{x\in S}\Gamma(x)\right|,
\]
i.e., the number of vertices adjacent to all elements of $S$. Note
that this generalises the notion of valency since for any vertex $x$,
$val\left(\left\{x\right\}\right) = val(x)$.

\begin{define}
Let $\Gamma=(V,E)$ be a graph, and $k\ge 1$ be an integer. Suppose
that for each set $S$ of at most $k$ vertices, the valency of $S$
depends only on the isomorphism class of the subgraph of $\Gamma$
induced by $S$. Then $\Gamma$ is called {\em $k$-isoregular}.
\end{define}

\begin{proposition}\label{prop:k-1iso}
For $k > 1$, $k$-isoregularity implies $(k-1)$-isoregularity.
\end{proposition}

\begin{proof}
Follows directly from the definition. If the valency of any set $S$ of
up to $k$ elements depends only on the isomorphism class of the
induced subgraph, this holds in particular for all sets of up to $k-1$
elements. 
\end{proof}

\begin{proposition}\label{prop:comp-iso}
A graph $\Gamma$ is $k$-isoregular if and only if its complement
$\overline\Gamma$ is $k$-isoregular.
\end{proposition}

\begin{proof}
Let $\Gamma$ be $k$-isoregular.
Let $S$ be a set of $k$ vertices in $\Gamma$. Since we have
$i$-isoregularity for $1\le i \le k$, we know the valency of every
subset of $S$. Using the Principle of Inclusion and Exclusion, we can
determine the number of vertices not adjacent to any element of $S$,
i.e., the valency of $S$ in the complement of $\Gamma$. Since the
calculation depends only on the isomorphism class of the subgraph of
$\Gamma$ induced by $S$, and since this holds for any $k$-subset $S$, 
$\overline\Gamma$ is $k$-isoregular.
\end{proof}

\begin{proposition}
A graph $\Gamma$ is $1$-isoregular if and only if it is regular. It is $2$-isoregular
if and only if it is strongly regular.
\end{proposition}

\begin{proof}
Since there is only one isomorphism class of graphs of order 1, and
since the valency of a singleton $\{x\}$ is the same as the valency of
the vertex $x$, a graph is $1$-isoregular if and only if each vertex has the same
valency, in other words, if and only if the graph is regular.

There are two isomorphism classes of graphs of order $2$, namely, edges
and non-edges. A graph is $2$-isoregular if and only if it is $1$-isoregular,
i.e., regular, and if the number of common neighbours of two distinct
vertices $x$ and $y$ depends only on whether $x$ is adjacent to
$y$. However, this is exactly the definition of strong regularity.
\end{proof}

For a graph $\Gamma$, a vertex $x$, and an integer $i$, we define the
{\em $i$-th subconstituent}\/ $\Gamma_i(x)$ as the subgraph of $\Gamma$
induced by all vertices at distance $i$ from $x$.

\begin{proposition}
For a strongly regular graph $\Gamma$, the following are equivalent:
\begin{itemize}
\item $\Gamma$ is $3$-isoregular.
\item The subconstituents $\Gamma_i(x)$, $i=1,2$, are strongly
  regular, with parameters which do not depend on the choice of $x$.
\end{itemize}
\end{proposition}

\begin{proof}
There are four isomorphism classes of graphs of order $3$, determined
uniquely by the number of edges they contain. We will denote these
classes by $\Delta_i$, where $0\le i\le 3$ is the number of edges.

Let $\Gamma$ be a $3$-isoregular graph. By Proposition
\ref{prop:k-1iso}, it is $2$-isoregular, and hence strongly regular,
with parameters $(v,k,\lambda,\mu)$.

Let $x$ be a vertex of $\Gamma$, and let $\Gamma_1=\Gamma(x)$, the
subgraph induced by the neighbours of $x$. $\Gamma_1$ is a regular
graph of order $k$ and valency $\lambda$. Let $y$ and $z$ be vertices
of $\Gamma_1$. Then the common neighbours of $y$ and $z$ in $\Gamma_1$
are exactly the common neighbours of $x,y,z$ in $\Gamma$.

If $y\sim z$, then $x,y,z$ induce a complete graph $\Delta_3$ in
$\Gamma$. Hence, $y$ and $z$ have $val(\Delta_3)$ neighbours in
$\Gamma_1$. Similarly, if $y\not\sim z$, $x,y,z$ induce the graph
$\Delta_2$ in $\Gamma$. Thus $y$ and $z$ have $val(\Delta_2)$
neighbours in $\Gamma_1$.  So we get that $\Gamma_1$ is strongly
regular, with parameters 
\begin{eqnarray*}
  v_1&=&k\\ 
  k_1&=&\lambda\\
  \lambda_1&=&val(\Delta_3)\\ 
  \mu_1&=&val(\Delta_2)).
\end{eqnarray*}
We get strong regularity for the second subconstituents by working
with the complement of $\Gamma$. More precisely, if $\Gamma$ is
3-isoregular, then so is the complement $\bar \Gamma$, by
Proposition~\ref{prop:comp-iso}. Then, by the argument above, the
first subconstituent of $\bar\Gamma$ is strongly regular; however,
this is precisely the complement of the second subconstituent of
$\Gamma$. 

Conversely, assume that $\Gamma$ is a strongly regular graph with
parameters $(v,k,\lambda,\mu)$, such that the subconstituents
$\Gamma_i(x)$, $i=1,2$,  are 
strongly regular 
with parameters $(v_i,k_i,
\lambda_i, \mu_i)$. By assumption, $\Gamma$ is $2$-isoregular, hence
we only need to check the graphs $\Delta_i$ of order $3$. Clearly, in $\Gamma$ we have
$val(\Delta_3)=\lambda_1$, and $val(\Delta_2) = \mu_1$.

Let $x,y,z$ be pairwise non-adjacent vertices in $\Gamma$. $y$ and $z$
have $\mu$ common neighbours in $\Gamma$; of these, $\mu_2$ are not
neighbours of $x$. Hence we get that $val(\Delta_0) = \mu-\mu_2$.

Similarly, if $y\sim z$, and both are non-adjacent to $x$, they have
$\lambda$ neighbours in $\Gamma$, and $\lambda_2$ neighbours in
$\Gamma_2$. Hence, $val(\Delta_1)=\lambda-\lambda_2$.

Altogether, we get that $\Gamma$ is $3$-isoregular.
\end{proof}

\section{The $t$-vertex condition}\label{sec:t-vertex}

Here, we discuss another way of generalising strong regularity. Let
$\Gamma$ be a graph of order $v$. Let $T$ be a graph of order $t$,
with two distinguished vertices $x_0$ and $y_0$. We will denote such a
triple $(T, x_0, y_0)$ as  a {\em graph type}. $x_0$ and $y_0$
will be called {\em fixed vertices}, the other vertices of $T$ will be
called {\em additional vertices}.

Let $x,y$ be vertices of $\Gamma$, and let $\Delta$ be an induced
subgraph of $\Gamma$ containing both $x$ and $y$. $\Delta$ is said to
be of {\em type $T$}\/ (with respect to $x$ and $y$) if there is an
isomorphism $\phi:T\to \Delta$ which maps $x_0$ to $x$ and $y_0$ to
$y$.

\begin{define}
Let $\Gamma$ be a graph, and let $t\ge 2$ be an integer. Suppose that
for any graph type $(T, x_0, y_0)$ of order at most $t$, and
for any pair of vertices $(x,y)$ of $\Gamma$, the number of subgraphs
of $\Gamma$ which are of type $T$ w.r.t.\ $x$ and $y$ depends only on
whether $x$ and $y$ are equal, adjacent, or non-adjacent. Then
$\Gamma$ is said to satisfy the {\em $t$-vertex condition.}
\end{define}

In other words, $\Gamma$ satisfies the $t$-vertex condition if we
cannot distinguish its edges (vertices, non-edges) by considering
subgraphs of order up to $t$.

\begin{proposition}
For $t>2$, the $t$-vertex condition implies the $(t-1)$-vertex
condition. 
\end{proposition}
\begin{proof}
Follows directly from the definition.
\end{proof}
\begin{proposition}
A graph $\Gamma$ satisfies the $2$-vertex condition if and only if it is
regular. It satisfies the $3$-vertex condition if and only if it is strongly
regular. 
\end{proposition}
 
\begin{proposition}
A rank 3 graph of order $v$ satisfies the $v$-vertex condition.
\end{proposition}

\begin{proof}
In a rank 3 graph, the automorphism group acts transitively on
vertices, (directed) edges and non-edges. Hence we cannot 
distinguish edges (resp. non-edges) on a combinatorial level.
\end{proof}

For growing $t$, the number of graph types to check increases very
quickly. However, it turns out that many of these checks are
redundant.

\begin{theorem}[\cite{Rei00}]\label{thm:t-vertex}
Let $\Gamma$ be a $k$-isoregular graph which satisfies the
$(t-1)$-vertex condition. In order to check that $\Gamma$ satisfies
the $t$-vertex condition  
it suffices to consider graph types in which each additional
vertex has valency at least $k+1$.
\end{theorem}

The idea behind the proof is the following: Suppose that $\Gamma$ does
not satisfy the $t$-vertex condition. Then there is a 
graph type $(T, x, y)$ such that the number of induced subgraphs of this
type does depend on the choice of $x$ and $y$. We may assume that $T$
is maximal with this property, in other words, for any graph type $T'$
where $T'$ properly containing $T$, the numbers are invariant. If $T$
contains an 
vertex $z$ (other than $x$ and $y$) of valency at most $k$, we can
delete $z$ and use the $(t-1)$ vertex condition to count the resulting
graph type; after that we use the $k$-isoregularity to recover the
number of graphs of type $(T,x,y)$. This contradicts the choice of $T$.

In fact, we will be using the following reformulation of Theorem
\ref{thm:t-vertex}: 
\begin{corollary}\label{cor:t-vertex}
Let $\Gamma$ be a graph which is $k$-isoregular, and which satisfies
the $(t-1)$-vertex condition, 
but not the $t$-vertex condition. Then there is a graph type
$(T,x,y)$ such that all vertices other than $x$ and $y$ have valency
at least $k+1$, and such that the number of graphs of type $(T,x,y)$
depends on the choice of $x$ and $y$ in $\Gamma$.
\end{corollary}

Finally, we present a complexity result related to the $t$-vertex
condition: 

\begin{theorem}
\label{thm:polynomial}
For a fixed integer $t$ and a graph $\Gamma$ of order $n$, the $t$-vertex
condition can be checked in time polynomial in $n$.
\end{theorem}

\begin{proof}
We may assume that $t\ge 3$, and that the $(t-1)$-vertex condition has
already been checked.

Fix a graph type $T=(\Delta, z_1, z_2)$ of order $t$ and two vertices $z_1', z_2'$ of
$\Gamma$. Suppose that either both $(z_1,z_2)$ is an edge in $\Delta$ and
$(z_1',z_2')$ is an edge in $\Gamma$, or both pairs are non-edges in their
respective graphs. Arbitrarily label the remaining vertices in
$\Delta$ by $z_3, \ldots, z_t$.

Now we consider all sequences $(z_3', \ldots, z_t')$ of distinct
vertices in $\Gamma$; there are fewer than $n^{t-2}$ of them. For each
sequence we may consider the labelled subgraph induced by the $z_i'$,
$1\le i\le t$. We can check in time proportional to $t^2$ whether
$\phi:z_i\mapsto z_i'$ is a graph isomorphism. Hence, in time
$O(n^{t-2}t^2)$ we can count the graphs of type $T$ with respect to
the given vertices $z_1'$ and $z_2'$.

There are fewer than $n^2$ possible pairs $(z_1', z_2')$ of
vertices. Repeating the count for all of them gives us the
numbers of graphs of type $T$ for all these pairs in time
$O(n^{t-2}\cdot n^2)= O(n^t)$, since $t$ is constant. Since there are
only finitely many graph types of order $t$, this proves the theorem.
\end{proof}

\section{Generalised quadrangles}
\label{sec:gq}

\begin{define}
Let $P$ be a finite set of {\em points}. Let $L$ be a set of
distinguished subsets of $P$, called {\em lines.}\/
 Suppose there are
integers $s$ and $t$ such that
\begin{itemize}
\item any two lines intersect in at most one point;
\item each line contains exactly $s+1$ points;
\item each point is contained in exactly $t+1$ lines.
\end{itemize}
Then $(P,L)$ is a {\em partial linear space} of order $(s,t)$.
\end{define}

We use the traditional geometric language: Two points are collinear if they
lie on a common line; two lines are concurrent if they intersect.

Given a partial linear space, we can define a graph as follows:
\begin{define}
Let $(P,L)$ be a partial linear space. Let $\Gamma$ be the graph with
vertex set $P$, two points being adjacent if they are collinear. Then
$\Gamma$ is called the {\em point graph} of $(P,L)$.
\end{define}

Generalised quadrangles are partial linear spaces which satisfy one
additional regularity property.
\begin{define}
Let $(P,L)$ be a partial linear space. Suppose that for each line $l$
and each point $P\notin l$, there is exactly one point on $l$
collinear to $P$. Then $(P,L)$ is a {\em generalised quadrangle} of
order $(s,t)$, or a $GQ(s,t)$ for short.
\end{define}

Below we survey a few classical results about generalized quadrangles.

\begin{theorem}
The point graph of a $GQ(s,t)$ is strongly regular, with parameters
\begin{eqnarray*}
v &=& (s+1)(st+1)\\
k &=& (s-1)t\\
\lambda &=& s-1\\
\mu &=& t+1
\end{eqnarray*}
\end{theorem}

We now look at subgraphs of the point graphs of generalised quadrangles.

\begin{theorem}[Cameron \cite{cameron75}]\label{thm:k4-e}
The point graph of a generalised quadrangle does not contain $K_4-e$
as an induced subgraph. Here, $K_4-e$ denotes the graph obtained by
removing one edge from the complete graph on 4 vertices.
\end{theorem}

\begin{define}
A {\em triad} in a $GQ(s,t)$ is a triple $\{x,y,z\}$ of pairwise
non-collinear points. A {\em center} of a triad is a point collinear
to all three points of the triad.
\end{define}

\begin{theorem}[see \cite{PaT84}]\label{thm:centers}
For a generalised quadrangle of order $(s,t)$, we have $s^2\ge
t$. Equality holds if and only if each triad has exactly $s+1$ centers.
\end{theorem}

\begin{corollary}\label{thm:gq-isoregular}
The point graph of a generalised quadrangle of order $(s,s^2)$ is
$3$-isoregular. 
\end{corollary}
\begin{proof}
The isomorphism class of a graph of order $3$ is uniquely determined
by the number of edges. By Theorem \ref{thm:centers}, each graph with
0 edges has $s+1$ common neighbours. By Theorem \ref{thm:k4-e}, each
graph with 2 edges has no common neighbours. 

If the graph has 1 edge, two of the points are connected by a
line. The third point has exactly one neighbour on this line, which is
the unique common neighbour of all three points. Thus, a graph with 1
edge has one common neighbour.

If the graph has 3 edges, it is complete, and hence contained in a
line. The common neighbours of the three points are the remaining $s-2$
points on this line.
\end{proof}
\section[Point graphs of generalised quadrangles]{Proof of Theorem 1}
\label{sec:t-vert-gq}
\subsection{Goals and strategy}

In this section we prove Theorem~\ref{thm:5vc-gq}, which states that the
point graph of a generalised quadrangle satisfies the $5$-vertex condition.
For generalised quadrangles of order $(s,s^2)$, this has been
previously shown by A.~Pasini (unpublished \cite{Pas91}). Although
the proof presented here is independent, the result in \cite{Pas91} showed that
generalised quadrangles yield a class of strongly regular graphs which should be further
investigated. Also, some of the techniques used by Pasini helped to obtain
the result presented here.

By Theorem \ref{thm:t-vertex} we have  to count graphs
with minimal valency 3 and order 5. In other words, if $x,y$ are the
fixed vertices, and $a,b,c$ the additional vertices, we need to
consider the graphs in which each of $a,b,c$ is non-adjacent to at most
one vertex.

If we take the complements of these graphs, we get that the valency of
$a,b,c$ is at most 1. If we discard a possible edge $(x,y)$, that
implies that the size of the graph, i.e., the number of its edges, is
at most 3. Thus, we use the following strategy:
\begin{enumerate}
\item Enumerate all graphs of order 5 and size $i=0,1,2,3$ (Subsection
  \ref{sub:graphs5});
\item Discard those graphs which cannot appear as subgraphs of
  generalised quadrangles (Subsection \ref{sub:easyDiscard});
\item For the few remaining graph types, check that their numbers do
  not depend on the choice of the edge or non-edge $(x,y)$ (Subsection
  \ref{sub:proof5vc-gq}). 
\end{enumerate}

\subsection{Graph types relevant to the 5-vertex condition}
\label{sub:graphs5}
We start by determining all graph types which need to be considered in
order to check whether a given graph satisfies the 5-vertex
condition. As stated above, the size of the complements of these graph
types is bounded by 3, in other words, the complements contain at most
3 edges.

In order to be sure not to miss anything, we will completely enumerate
these complements. In the
following we enumerate the relevant graphs up to isomorphism
under the group $S(\{x,y\})\times S(\{a,b,c\})$ of order 12, counting
also the cardinalities of the related isomorphism classes. We will
not consider the set $\{x,y\}$ as a possible edge. This leaves
$\binom{5}{2}-1 = 9$ possible edges.

We denote graphs by a list of its edges.
\begin{description}
\item[Graphs of size 0] $\quad$

 The empty graph is uniquely determined by its
  size. 
\item[Graphs of size 1] $\quad$

\begin{enumerate}
\item[a.] $\{x,a\}$, $Aut = \left<(b,c)\right>, |Aut| = 2$, length of orbit: 6.
\item[b.] $\{a,b\}$, $Aut = \left<(a,b), (x,y)\right>, |Aut| = 4$, length of orbit: 3.
\end{enumerate}

This accounts for $6+3=9=\binom{9}{1}$ graphs.
\item[Graphs of size 2] $\quad$

\begin{enumerate}
\item[a.] $\{x,a\},\{x,b\}$, $Aut = \left<(a,b)\right>, |Aut| = 2$, length
  of orbit: 6.
\item[b.] $\{x,a\},\{y,b\}$, $Aut = \left<(a,b)(x,y)\right>, |Aut|=2$,
  length of orbit: 6. 
\item[c.] $\{x,a\},\{b,c\}$, $Aut = \left<(b,c)\right>, |Aut|=2$,
  length of orbit: 6. 
\item[d.] $\{x,a\},\{y,a\}$, $Aut = \left<(x,y),(b,c)\right>, |Aut|=4$,
  length of orbit: 3. 
\item[e.] $\{x,a\},\{a,b\}$, $Aut = \left<e\right>, |Aut|=1$,
  length of orbit: 12. 
\item[f.] $\{a,b\},\{b,c\}$, $Aut = \left<(a,c),(x,y)\right>, |Aut|=4$,
  length of orbit: 3. 
\end{enumerate}
This accounts for $6+6+3+12+6+3=36=\binom{9}{2}$ graphs.
\item[Graphs of size 3] $\quad$

\begin{enumerate}
\item[a.] $\{x,a\},\{x,b\},\{x,c\}$, $Aut=S(\{a,b,c\}$, $|Aut| = 6$,
  length of orbit: 2.
\item[b.] $\{x,a\},\{x,b\},\{y,c\}$, $Aut=\left<(a,b)\right>$, $|Aut| = 2$,
  length of orbit: 6.
\item[c.] $\{x,a\},\{x,b\},\{y,a\}$, $Aut=\left<e\right>$, $|Aut| = 1$,
  length of orbit: 12.
\item[d.] $\{x,a\},\{x,b\},\{a,b\}$, $Aut=\left<(a,b)\right>$, $|Aut| = 2$,
  length of orbit: 6.
\item[e.] $\{x,a\},\{x,b\},\{a,c\}$, $Aut=\left<e\right>$, $|Aut| = 1$,
  length of orbit: 12.
\item[f.] $\{x,a\},\{y,a\},\{a,b\}$, $Aut=\left<(x,y)\right>$, $|Aut| = 2$,
  length of orbit: 6.
\item[g.] $\{x,a\},\{y,a\},\{b,c\}$, $Aut=\left<(x,y),(b,c)\right>$, $|Aut| = 4$,
  length of orbit: 3.
\item[h.] $\{x,a\},\{y,b\},\{a,b\}$, $Aut=\left<(a,b)(x,y)\right>$, $|Aut| = 2$,
  length of orbit: 6.
\item[i.] $\{x,a\},\{y,b\},\{a,c\}$, $Aut=\left<e\right>$, $|Aut| = 1$,
  length of orbit: 12.
\item[j.] $\{x,a\},\{a,b\},\{a,c\}$, $Aut=\left<(b,c)\right>$, $|Aut| = 2$,
  length of orbit: 6.
\item[k.] $\{x,a\},\{a,b\},\{b,c\}$, $Aut=\left<e\right>$, $|Aut| = 1$,
  length of orbit: 12.
\item[l.] $\{a,b\},\{a,c\},\{b,c\}$, $Aut=\left<(x,y),(a,b),(a,b,c)\right>$, $|Aut| = 12$,
  length of orbit: 1.
\end{enumerate}
This accounts for $2 +4\cdot 12 + 5\cdot 6 + 3+1 = 84 = \binom{9}{3}$
graphs. This check sum confirms that our enumeration is complete. 
\end{description}

The following graphs can be discarded because one of the additional
vertices has valency greater than 1: 2.d--f, 3.c--l

This can be summarised as follows (recall that above we enumerated the
complements of the relevant graphs):
\begin{theorem}
If a graph $\Gamma$ satisfies the 4-vertex condition, then to check
the 5-vertex condition it is sufficient to count the graphs of the eight
types given in Table \ref{tbl:5-cond-types}. Here, a dashed line
indicates an optional edge connecting the fixed vertices.
\end{theorem}

\begin{table}
\begin{center}
\begin{picture}(0,0)%
\includegraphics{types.pstex}%
\end{picture}%
\setlength{\unitlength}{2763sp}%
\begingroup\makeatletter\ifx\SetFigFont\undefined%
\gdef\SetFigFont#1#2#3#4#5{%
  \reset@font\fontsize{#1}{#2pt}%
  \fontfamily{#3}\fontseries{#4}\fontshape{#5}%
  \selectfont}%
\fi\endgroup%
\begin{picture}(6766,6608)(218,-5986)
\end{picture}

\end{center}
\caption{Types to check for 5-vertex condition}
\label{tbl:5-cond-types}
\end{table}

\subsection{Easily discarded cases}
\label{sub:easyDiscard}
We now proceed to apply the results obtained above to point graphs of
generalised quadrangles. In this case, most of these graph types can be
discarded since they contain a subgraph $K_4-e$, which cannot happen
in a generalised quadrangle (see Theorem \ref{thm:k4-e}).
\begin{lemma}
In a generalised quadrangle, the following graph types do not appear:
1a, 1b, 2b, 2c, 3b.
\end{lemma}
\begin{proof}
In the pictures of Table \ref{tbl:5-cond-types}, we denote the fixed
vertices left to right by $x,y$, and the additional vertices left to
right by $a,b,c$. For each of the graph types, we note a set of
vertices which induces a $K_4-e$:
\begin{description}
\item[1a:] $x,a,b,c$
\item[1b:] $x,a,b,c$
\item[2b:] $x,a,b,c$
\item[2c:] $y,a,b,c$
\item[3b:] $y,a,b,c$
\end{description}
So these types do not occur, independent  of whether $x$ and $y$ are
adjacent. 
\end{proof}

\subsection{The remaining cases}
\label{sub:proof5vc-gq}
This leaves us to check the following types: 0, 2a, 
3a.
 
We will now show that the numbers for these subgraphs are uniquely
determined by the axioms of generalised quadrangles.
We will arrange these verifications in three lemmas, one for
each type.

\begin{lemma}
The number of graphs of type 0 w.r.t. $x$ and $y$ is $\binom{s-1}{3}$
if $x\sim y$, 0 otherwise.
\end{lemma}
\begin{proof}
If $x\not \sim y$, then the graph contains a $K_4-e$. Otherwise, each
additional vertex is adjacent to both $x$ and $y$ and hence lies on
the line connecting them. Each set of three points on this line yields
a graph of type 0.
\end{proof}

\begin{lemma}
The number of graphs of type 2a is 0 if $x\sim y$. Otherwise, it is
$(t+1)\binom{s-1}{2}$. 
\end{lemma}
\begin{proof}
If $x\sim y$, then $\{x,y,b,c\}$ induces a $K_4-e$. Let $x\not\sim
y$. The points $\{y,a,b,c\}$ form a clique, hence lie on a line. Thus
to find such a graph, we choose a line $l$ through $y$, with $t+1$
possibilities. The point $c$ is the unique neighbour of $x$ on $l$; for
$a$ and $b$ we can choose any two of the remaining points on $l$.
\end{proof}
\begin{lemma}
The number of graphs of type 3a is $t\binom{s}{3}$ if $x\sim y$,
$(t+1) \binom{s-1}{3}$ otherwise.
\end{lemma}
\begin{proof}
The vertices $\{y,a,b,c\}$ form a clique, hence are collinear. If
$x\sim y$, we can choose any line through $y$ but not through $x$;
there are $t$ such lines. We need to choose 3 other points on this
line; this gives us $\binom{s}{3}$ choices.

If $x\not\sim y$, then we can choose any line through $y$ ($t+1$
possibilities) and then choose any three points on that line that are
not adjacent to $x$, giving $\binom{s-1}{3}$ choices.
\end{proof}

We see that for all the graph types we have to check, the numbers do
not depend on the particular choice of $x$ and $y$. Together with 
Theorem \ref{thm:t-vertex}, this proves Main Theorem \ref{thm:5vc-gq}.

\section{Proof of Theorem 2}
\subsection{Goals and preliminaries}

For the remainder of the text, we concentrate on point graphs of
generalised quadrangles of order $(s, s^2)$. Thus, let $s$ be fixed,
and let $\Gamma$ be the point graph of such a generalised quadrangle. 

Suppose that $\Gamma$ does not satisfy the $t_0$-vertex condition
for some value of $t_0$. We may assume that $t_0$ is minimal with this
property. Thus there is a graph type $(T, x_0, y_0)$ of order
$t_0$ such
that the number of graphs of type $T$ with respect to $x$ and $y$
does depend on the choice of $x$ and $y$. We may further assume that
$T$ is maximal, i.e., adding an edge to $T$ leads to a graph type
such that the corresponding number does not depend on the choice of
the fixed vertices. 

We will try to find a lower bound on $t_0$; this will yield the result
that $\Gamma$ satisfies the $(t_0-1)$-vertex condition.

First, let us collect some properties of $T$.
\begin{proposition}
We may assume that the  valency of any additional vertex in $T$ is
at least 4. 
\end{proposition}
\begin{proof}
Since $\Gamma$ is 3-isoregular by Corollary \ref{thm:gq-isoregular},
this follows from Corollary \ref{cor:t-vertex}.
\end{proof}

\begin{lemma}
$T$ does not contain an induced subgraph isomorphic to $K_4-e$.
\end{lemma}
\begin{proof}
Follows from Theorem \ref{thm:k4-e} and the fact that $T$ is an
induced subgraph of $\Gamma$.
\end{proof}

\begin{proposition}\label{prop:complete-empty}
Let $X$ be the set of vertices adjacent to both $x$ and $y$ in
$T$. If $x\sim y$, then the subgraph induced by $X$ is complete. If
$x\not\sim y$, then the subgraph induced by $X$ is empty.
\end{proposition}
\begin{proof}
Let $w$ and $z$ be two distinct vertices both adjacent to each of $x$ and
$y$. If exactly one of $(x,y), (w,z)$ is an edge in $T$, then these
four vertices induce a graph isomorphic to $K_4-e$, contradicting the
lemma above.
\end{proof}

Let $S$ be the subgraph obtained from $T$ by deleting $x$
and $y$, and all edges incident with either $x$ or $y$, see Figure
\ref{figure:S_T}. Then we have
\begin{itemize}
\item $S$ has order $t_0-2$;
\item $S$ has minimal valency at least 2;
\item The vertices of valency 2 in $S$ induce either a complete graph
  or an empty graph.
\item For any maximal clique $C$ in $S$ and any vertex $z$ not in $C$,
  $z$ has at most one neighbour in $C$.
\end{itemize}
\begin{figure}
\begin{center}
\input{S_T.pstex_t}
\caption{The graphs $S$ and $T$}
\label{figure:S_T}
\end{center}
\end{figure}

We now look at the possible subgraphs induced by $S$. In
Subsection~\ref{sec:complete}, we look at the case that $S$ induces a
complete graph. In Subsection~\ref{sec:not-complete}, we consider the
other case.

\subsection{$S$ is complete}\label{sec:complete}

We first consider the case that $S$ induces a complete graph in
$T$. In this case, we can determine the number of graphs of type $T$,
independent of the size of $S$.

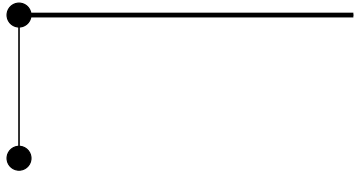
\begin{theorem}
Let $\Gamma$ be the point graph of a $GQ(s, s^2)$. Let $(T,x, y)$
be a graph type of arbitrary size $t_0$, such that the additional
vertices induce a complete graph $S$ in $T$. Then the number of graphs
of type $(T, x_0, y_0)$ with respect to vertices $x$ and $y$ does not
depend on the choice of $x$ and $y$ in $\Gamma$.
\end{theorem}

\begin{proof}
If $S$ is a complete graph, then $S$ is contained in a line, say
$l$. Either $x\in l$ and thus is collinear to all $t_0-2$ points in
$S$, or it has at most one neighbour in $S$. The same holds for
$y$. Let $d_x=|T(x)\cap S|$ be the number of neighbours of $x$ in $S$,
and similar $d_y=|T(y)\cap S|$. We have that $d_x,
d_y\in\{0,1,t_0-2\}$, and without loss of generality we may assume that
  $d_x\ge d_y$. We get six possibilities for the pair $(d_x, d_y)$,
and in each case we can determine the number of graphs of type $T$
with respect to any pair $(x,y)$ of distinct vertices in $\Gamma$.

\begin{description}
\item[$(d_x, d_y) = (t_0-2,t_0-2)$]:

  \begin{center}
    \begin{picture}(0,0)%
\includegraphics{line1.pstex}%
\end{picture}%
\setlength{\unitlength}{3947sp}%
\begingroup\makeatletter\ifx\SetFigFont\undefined%
\gdef\SetFigFont#1#2#3#4#5{%
  \reset@font\fontsize{#1}{#2pt}%
  \fontfamily{#3}\fontseries{#4}\fontshape{#5}%
  \selectfont}%
\fi\endgroup%
\begin{picture}(2266,166)(818,456)
\end{picture}

  \end{center}
  In this case, both $x$ and $y$ are contained in
  $l$; in particular, $x\sim y$. Thus there are
  $\binom{s-1}{|S|}$ choices   for the additional vertices.
\item[$(d_x, d_y) =(t_0-2, 1)$]:
  Here, $x\in l$, $y\notin l$, and the unique
  neighbour $z$ of $y$ on $l$ lies in $S$. We have to distinguish the
  cases $x=z$ and $x\neq z$.
  
  \begin{center}
    \begin{picture}(0,0)%
\includegraphics{line8.pstex}%
\end{picture}%
\setlength{\unitlength}{3947sp}%
\begingroup\makeatletter\ifx\SetFigFont\undefined%
\gdef\SetFigFont#1#2#3#4#5{%
  \reset@font\fontsize{#1}{#2pt}%
  \fontfamily{#3}\fontseries{#4}\fontshape{#5}%
  \selectfont}%
\fi\endgroup%
\begin{picture}(2216,1066)(818,-444)
\end{picture}

  \end{center}
  If $x\sim y$, we can choose any line $l$ through $x$ but not through
  $y$, which gives $s^2$ possibilities. Then we need to choose a
  subset $S\subseteq l\setminus\{x\}$. Hence, altogether there are
  $s^2\binom{s}{t_0-2}$ such graphs.
  
  
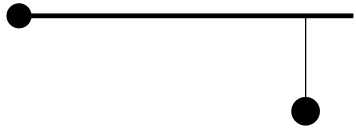
\begin{center}
    \begin{picture}(0,0)%
\includegraphics{line2.pstex}%
\end{picture}%
\setlength{\unitlength}{3947sp}%
\begingroup\makeatletter\ifx\SetFigFont\undefined%
\gdef\SetFigFont#1#2#3#4#5{%
  \reset@font\fontsize{#1}{#2pt}%
  \fontfamily{#3}\fontseries{#4}\fontshape{#5}%
  \selectfont}%
\fi\endgroup%
\begin{picture}(2216,781)(818,-159)
\end{picture}

  \end{center}
  If $x\neq z$, then $x\not\sim
  y$. In order to find such a graph, we have to first choose the line
  $l$ through $x$, which gives $s^2+1$ possibilities. The point $z\in l$
  is uniquely determined by $y$, and we need to choose the $t_0-3$
  remaining vertices of $S$ on $l$. Altogether, there are
  $(s^2+1)\binom{s-1}{t_0-3}$ such graphs.
\item[$(d_x, d_y) =(t_0-2,0)$]:
  Here, $x\in l$, and the unique neighbour $z$ of $y$ on
  $l$ does not lie in $S$. In particular, $x\not\sim y$.
  
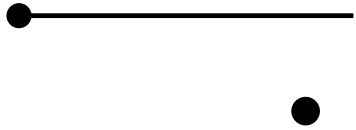
  \begin{center}
    \begin{picture}(0,0)%
\includegraphics{line3.pstex}%
\end{picture}%
\setlength{\unitlength}{3947sp}%
\begingroup\makeatletter\ifx\SetFigFont\undefined%
\gdef\SetFigFont#1#2#3#4#5{%
  \reset@font\fontsize{#1}{#2pt}%
  \fontfamily{#3}\fontseries{#4}\fontshape{#5}%
  \selectfont}%
\fi\endgroup%
\begin{picture}(2216,781)(818,-159)
\end{picture}

  \end{center}
  We can choose any line $l$ through $x$, and then a
  subset $S\subseteq l$ avoiding both $x$ and the unique neighbour of
  $y$ on $l$. Thus, the total number of such graphs is
  $(s^2+1)\binom{s-1}{t_0-2}$. 
\item [$(d_x, d_y) =(1,1)$]:
  Neither $x$ nor $y$ lie on $l$, and their unique neighbours $z_x$ and
  $z_y$ lie in $S$. We need to distinguish four cases, according as
  $x\sim y$, and $z_x=z_y$.
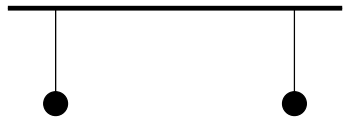
  \begin{description}
  \item[ $x\not\sim y$, $z_x\neq z_y$]:
    
    \begin{center}
      \begin{picture}(0,0)%
\includegraphics{line4.pstex}%
\end{picture}%
\setlength{\unitlength}{3947sp}%
\begingroup\makeatletter\ifx\SetFigFont\undefined%
\gdef\SetFigFont#1#2#3#4#5{%
  \reset@font\fontsize{#1}{#2pt}%
  \fontfamily{#3}\fontseries{#4}\fontshape{#5}%
  \selectfont}%
\fi\endgroup%
\begin{picture}(2166,716)(868,-144)
\end{picture}

    \end{center}
    Choose any line $l_1$ through $x$ ($s^2+1$ possibilities). On
    $l_1$ we take any other point $z_x$ which is not adjacent to $y$
    ($s-1$ possibilities). Through $z_x$ we take any line 
    $l\neq l_1$ ($s^2$ choices). On $l$, $z_x$ and $z_y$ are fixed,
    so we have to choose $|S|-2$ additional points. Thus the total
    number of such graphs is
    \[
    s^2(s^2+1)(s-1)\binom{s-1}{|S|-2}.
    \]
  \item[ $x\sim y$, $z_x\neq z_y$]:

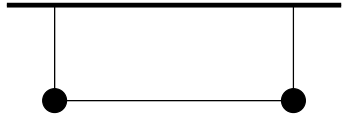
    \begin{center}
      \begin{picture}(0,0)%
\includegraphics{line11.pstex}%
\end{picture}%
\setlength{\unitlength}{3947sp}%
\begingroup\makeatletter\ifx\SetFigFont\undefined%
\gdef\SetFigFont#1#2#3#4#5{%
  \reset@font\fontsize{#1}{#2pt}%
  \fontfamily{#3}\fontseries{#4}\fontshape{#5}%
  \selectfont}%
\fi\endgroup%
\begin{picture}(2166,716)(868,-144)
\end{picture}

    \end{center}
    Take any line $l_1$ through $x$ which does not contain $y$ ($s^2$
    choices). Choose $z_x\neq x$ on $l_1$ ($s$ choices). By
    construction, $z_x\not\sim y$.  Take any line $l\neq l_1$ through
    $z_x$ ($s^2$ choices); this determines $z_y$.  Again, we need to
    choose the remaining $|S|-2$ points on $l$. The total number is
    \[
    s^5\binom{s-1}{|S|-2}.
    \]
  \item[$x\not\sim y$, $z_x= z_y$]:

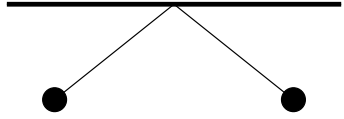
    \begin{center}
      \begin{picture}(0,0)%
\includegraphics{line5.pstex}%
\end{picture}%
\setlength{\unitlength}{3947sp}%
\begingroup\makeatletter\ifx\SetFigFont\undefined%
\gdef\SetFigFont#1#2#3#4#5{%
  \reset@font\fontsize{#1}{#2pt}%
  \fontfamily{#3}\fontseries{#4}\fontshape{#5}%
  \selectfont}%
\fi\endgroup%
\begin{picture}(2166,716)(868,-144)
\end{picture}

    \end{center}
    We take any common neighbour $z$ of $x$ and $y$ ($\mu$
    choices). Through $z$ there are $s^2-1$ lines $l$ which do not contain
    $x$ or $y$. On $l$, we need to choose $|S|-1$ additional
    points. Hence, we get a total number of
    \[
    \mu(s^2-1)\binom{s}{|S|-1}.
    \]
  \item[ $x\sim y$, $z_x= z_y$]:

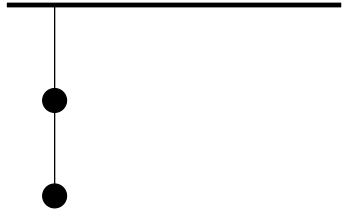
    \begin{center}
      \begin{picture}(0,0)%
\includegraphics{line12.pstex}%
\end{picture}%
\setlength{\unitlength}{3947sp}%
\begingroup\makeatletter\ifx\SetFigFont\undefined%
\gdef\SetFigFont#1#2#3#4#5{%
  \reset@font\fontsize{#1}{#2pt}%
  \fontfamily{#3}\fontseries{#4}\fontshape{#5}%
  \selectfont}%
\fi\endgroup%
\begin{picture}(2166,1316)(868,-744)
\end{picture}

    \end{center}
    Here, the points $x$, $y$, and $z_x$ are pairwise adjacent, hence
    they are collinear. Let $l_1$ be the line through $x$ and $y$. We
    can choose any third point $z_x$ on $l_1$ ($s-1$ possibilities) and
    any line $l\neq l_1$ through $z_x$ ($s^2$ possibilities). Finally,
    we need to take $|S|-1$ additional points on $l$. The total number
    is
    \[
    s^2(s-1)\binom{s}{|S|-1}.
    \]
  \end{description}
\item [$(d_x, d_y) =(1,0)$]:
  Again, we distinguish the cases $x\sim y$ and $x\not\sim y$.
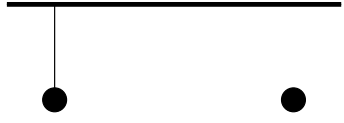
  \begin{center}
    \begin{picture}(0,0)%
\includegraphics{line6.pstex}%
\end{picture}%
\setlength{\unitlength}{3947sp}%
\begingroup\makeatletter\ifx\SetFigFont\undefined%
\gdef\SetFigFont#1#2#3#4#5{%
  \reset@font\fontsize{#1}{#2pt}%
  \fontfamily{#3}\fontseries{#4}\fontshape{#5}%
  \selectfont}%
\fi\endgroup%
\begin{picture}(2166,716)(868,-144)
\end{picture}

  \end{center}
  If $x\not\sim y$, we can take any line $l_1$ through $x$, and choose
  $x\neq z_x\not\sim y$. Through $z_x$ we take any line $l\neq l_1$,
  and on $l$ we need $|S|-1$ additional points, avoiding the neighbour
  of $y$. The total number of choices is
  \[
  (s^2+1)(s-1)s^2\binom{s-1}{|S|-1}.
  \]

  
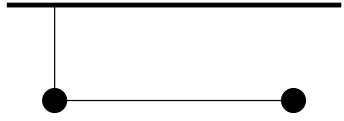
\begin{center}
    \begin{picture}(0,0)%
\includegraphics{line9.pstex}%
\end{picture}%
\setlength{\unitlength}{3947sp}%
\begingroup\makeatletter\ifx\SetFigFont\undefined%
\gdef\SetFigFont#1#2#3#4#5{%
  \reset@font\fontsize{#1}{#2pt}%
  \fontfamily{#3}\fontseries{#4}\fontshape{#5}%
  \selectfont}%
\fi\endgroup%
\begin{picture}(2166,716)(868,-144)
\end{picture}

  \end{center}
  If $x\sim y$, we choose $x\in l_x \not\ni y$. On $l_x$ we take any
  point $z_x\neq x$. Through $z_x$ we take any line $l\neq l_x$, and
  on $l$ we need $|S|-1$ more points, none of which is adjacent to
  $y$. The total number is
  \[
  s^2\cdot s \cdot s^2\binom{s-1}{|S|-1}.
  \]
\item [$(d_x, d_y) =(0,0)$]:

  
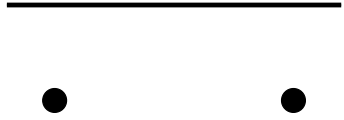
\begin{center}
    \begin{picture}(0,0)%
\includegraphics{line7.pstex}%
\end{picture}%
\setlength{\unitlength}{3947sp}%
\begingroup\makeatletter\ifx\SetFigFont\undefined%
\gdef\SetFigFont#1#2#3#4#5{%
  \reset@font\fontsize{#1}{#2pt}%
  \fontfamily{#3}\fontseries{#4}\fontshape{#5}%
  \selectfont}%
\fi\endgroup%
\begin{picture}(2166,716)(868,-144)
\end{picture}

  \end{center}
  If $x\not\sim y$, we can divide all lines in the GQ into three different
  parts. There are $2(s^2+1)$ lines which pass through either $x$ and
  $y$. On each of the remaining lines, both $x$ and $y$ have one
  common neighbour each. Let $L_1$ be the set of lines where these
  neighbours coincide, and $L_2$ the set of lines where the neighbours
  are distinct. 

  Now $x$ and $y$ have $\mu$ common neighbours. Through each such neighbour
  there are $s^2-1$ lines which do not pass through either $x$ or
  $y$. All these lines are distinct, since otherwise, $x$ would have
  two neighbours on one line. Hence, we get that $|L_1|=\mu(s^2-1)$,
  and thus $|L_2|=l-|L_1|-2(s^2+1)$, where $l$ is the total number of
  lines. 

  In order to obtain a graph depicted above, we take any line from
  $L_1\cup L_2$. On this line, we need to choose $|S|$ points avoiding
  the neighbours of $x$ and $y$. Thus, the total number of such graphs
  is
  \[
  |L_1|\binom{s}{|S|} + |L_2|\binom{s-1}{|S|}.
  \]
  
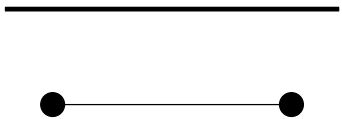
  \begin{center}
    \begin{picture}(0,0)%
\includegraphics{line10.pstex}%
\end{picture}%
\setlength{\unitlength}{3947sp}%
\begingroup\makeatletter\ifx\SetFigFont\undefined%
\gdef\SetFigFont#1#2#3#4#5{%
  \reset@font\fontsize{#1}{#2pt}%
  \fontfamily{#3}\fontseries{#4}\fontshape{#5}%
  \selectfont}%
\fi\endgroup%
\begin{picture}(2166,716)(868,-144)
\end{picture}

  \end{center}
  If $x\sim y$, we distinguish four types of lines: The line
  $l_0$ connecting $x$ and $y$; the set $L_1$ of lines intersecting
  $l_0$ in either $x$ or $y$; the set $L_2$ of lines intersecting
  $l_0$ in other points; and finally the set $L_3$ of lines skew to $l_0$.

  Clearly, $|L_1|=2s^2$.
  There are $s-1$ additional points on $l_0$; through each of them,
  there are $s^2$ lines distinct from $l_0$. All these lines are
  distinct, since otherwise, two lines would intersect in more than
  one point. Hence we get that $|L_2|=s^2(s-1)$, and hence,
  $|L_3|=l-1-|L_1| - |L_2|$. 

  Now, by a similar reasoning as above, we get that the total number
  of graphs is
  \[
  |L_2|\binom{s}{|S|} + |L_3|\binom{s-1}{|S|}.
  \]
  
\end{description}
This completes the proof.
\end{proof}

\subsection{$S$ is not complete}\label{sec:not-complete}

Recall that $S$ is a set of $t_0-2$ vertices in the graph $T$ of order
$t_0$. 

\begin{proposition}
$S$ does not contain a clique of size $t_0-3$.
\end{proposition}
\begin{proof}
If $S$ contains only one point $z$ not contained in a maximal clique $C$,
then $z$ has to have two neighbours in $C$, which is  a contradiction.
\end{proof}

\begin{proposition}
$S$ does not contain a clique of size $t_0-4$.
\end{proposition}

\begin{proof}
Suppose that $S$ contains such a clique $C'$, and that it contains two more
vertices, $w$ and $z$. By the condition on the minimal valency, $w$
and $z$ are adjacent, and both have  exactly one neighbour in
$C'$. Moreover, in $T$, they are both adjacent to each of $x$ and $y$,
forcing $x\sim y$. Thus, in $T$, we have two cliques, $C'$ of size
$t_0-4$, and $C=\{x,y,w,z\}$ of size 4.

Let $l$ be the line in the generalised quadrangle containing $C$, and
let $l'$ be the line connecting $x$ and $y$. The latter is uniquely
determined. Denote the unique neighbours of $x,y,w,z$ on $l'$ by $x',
y', w', z'$ respectively. 

If $w'=z'$, then $l$ and $l'$ intersect, and $x'=y'=w'\in C'$. In this
case, we can count the graphs as follows: We select $w$ and $z$ on the
line connecting $x$ and $y$. The intersection point $x'$
has to lie on $l\setminus C$, which gives $s-3$ choices. There are $s^2$
other lines through $x'$; we choose one of them and then select
$t_0-5$ additional points on this line. Altogether,  we get
$\binom{s-1}{2} (s-3) s^2 \binom{s}{t_0-5}$ such graphs.

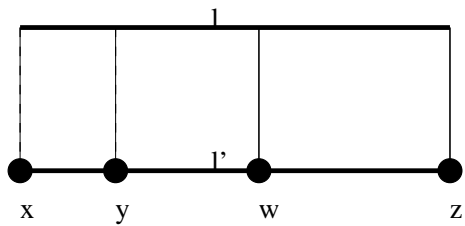
\begin{center}
\begin{picture}(0,0)%
\includegraphics{xywz.pstex}%
\end{picture}%
\setlength{\unitlength}{3947sp}%
\begingroup\makeatletter\ifx\SetFigFont\undefined%
\gdef\SetFigFont#1#2#3#4#5{%
  \reset@font\fontsize{#1}{#2pt}%
  \fontfamily{#3}\fontseries{#4}\fontshape{#5}%
  \selectfont}%
\fi\endgroup%
\begin{picture}(2891,1375)(518,-1001)
\end{picture}

\end{center}

If $w'\neq z'$, then $l$ and $l'$ do not intersect, and hence
$x',y',w',z'$ are all distinct. We choose the points $w$ and $z$ on
$l$, and a line $l'$ not intersecting $l$. The points $w'$ and $z'$
are uniquely determined; we can choose the remaining points according
to whether or not $x'$ and $y'$ are in $C'$. In each case, we can
count easily how many graphs we obtain.

Thus, the number of graphs of type $T$ with respect to $x$ and $y$
does not depend on the choice of $x$ and $y$, which is a contradiction
to the choice of $T$.
\end{proof}

Collecting what we have so far, we get:
\begin{corollary}
  $S$ is a graph of order $t_0-2$ of minimal valency at least $2$ not containing a
  $t_0-4$ clique.
\end{corollary}

\begin{corollary}
$t_0\ge 7$.
\end{corollary}
\begin{proof}
By the minimal valency condition, $S$ contains an edge, i.e., a
2-clique. Thus $t_0-4>2$.
\end{proof}

We now consider how many vertices in $S$ can have valency $2$.

\begin{proposition}
If $x\sim y$, then there are at most three vertices in $S$ of valency $2$.
\end{proposition}
\begin{proof}
In this case, the vertices of valency $2$ induce a complete graph by
Proposition~\ref{prop:complete-empty}.
\end{proof}

\begin{proposition}
If $x\not\sim y$, then there are at least two vertices with valency at
least $3$.
\end{proposition}
\begin{proof}
In this case, the vertices of valency 2 induce an empty graph by
Proposition~\ref{prop:complete-empty}. Each
one has to have two neighbours, which necessarily have a valency
greater than $2$.
\end{proof}

\begin{proposition}
Let $x\not\sim y$. Let $z$ be a vertex of valency 3 in $S$. Then $z$
can be adjacent to at most one vertex of valency 2.
\end{proposition}

\begin{proof}
Let $z$ be a vertex of valency 3. In $T$ it has to be adjacent to at
least one of $x$ and $y$; let us assume that $z\sim x$. Let $v$ and
$w$ be two vertices of valency 2 adjacent to $z$; this implies
$v\not\sim w$. Both $v$ and $w$ have to be adjacent to $x$, which
implies that $\{x,z,v,w\}$ induces a $K_4-e$.
\end{proof}

\begin{proposition}
The case $t_0=7$ is impossible.
\end{proposition}

\begin{proof}
Assume that $t_0=7$, i.e., $S$ has $t_0-2=5$ vertices. Then $S$ does not
contain a 3-clique. Let $z$ be a vertex with maximal valency in
$S$. Then $z$ has at least 3 neighbours. These neighbours must be mutually
non-adjacent to avoid a 3-clique. Since they have valency at least 2,
they are all adjacent to one additional vertex. Since this accounts
for 5 vertices, there are no additional vertices or edges in $S$.

However, now $z$ has valency 3, and it is adjacent to three
non-adjacent vertices of valency 2, which is a contradiction.
\end{proof}

Thus we have proved Theorem~\ref{thm:gqss2}.

\begin{corollary}
The constant $t_0$ in Klin's Conjecture is at least 8.
\end{corollary}

\begin{proof}
There are generalised quadrangles of order $(s,s^2)$ for some prime
powers $s\ge 5$   with 
intransitive automorphism groups, see \cite{Pay90c}.
\end{proof}

\section{Towards the 8-vertex condition}

For the 8-vertex condition, we have to consider a 6-vertex graph $S$
satisfying all the conditions stated in the previous section. A
computer search has been performed, and the result is that there are 5
different graphs to consider: The complete bipartite graphs $K_{3,3}$
and $K_{4,2}$; the graph $K_{3,3}-e$ obtained by removing an edge
from the complete bipartite graph; the triangular prism $K_3\circ
K_2$, and the graph obtained from the prism by removing one edge from
the prism which is not contained in a 3-clique. These graphs are shown
in Figure \ref{tab:graphs8vc}.

\begin{figure}
\begin{center}
\input{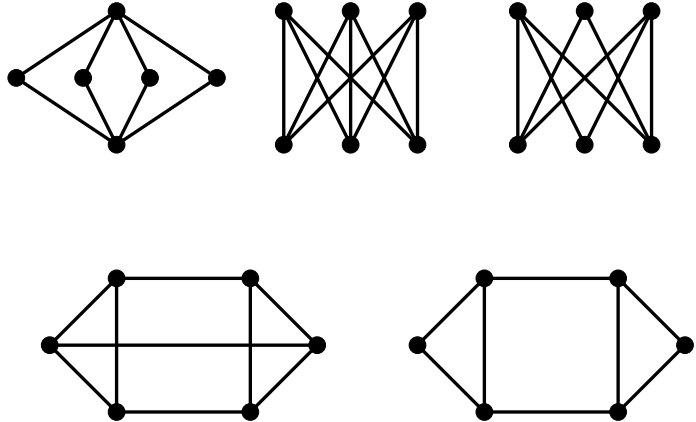}
\caption{Graphs relevant to the 8-vertex condition}
\label{tab:graphs8vc}
\end{center}
\end{figure}

It can be shown that the number of graphs of the types related to the
prism can be determined in a generalised quadrangle of order
$(s,s^2)$. However, for the bipartite graph types this is not the
case:

\begin{proposition}
There is a $GQ(5,25)$ whose point graph does not satisfy the
$8$-vertex condition. 
\end{proposition}

This result was obtained via computer search using known
generalised quadrangles. The quadrangle in question is obtained
using the set of matrices
\[
\left\{
\left(
\begin{matrix}
t&3t^2\\
0&3t^3
\end{matrix}
\right)\middle|
t\in GF(5)\right\},
\]
 see \cite{PaT84} and \cite{Pay90c} for details of
this construction. Its automorphism group has  several orbits on
edges,
which can be 
distinguished by counting complete bipartite graphs $K_{4,4}$
containing a given edge.

\section{Discussion}

\label{sec:history}

In this section we consider a number of topics which are naturally
related to the main part of the paper but which were not immediately
required for our presentation.

Nevertheless, we believe that a wider picture may be helpful for the
reader. Also credits should be given to many other researchers who
were dealing with the $t$-vertex condition, either explicitly or
implicitly. 

It is also worth mentioning that a first draft of the result given
above appeared in the authors thesis \cite{Rei03}, prepared at the
University of Delaware.

\subsection{More about strongly regular graphs}
\label{sub:srg}
The classical concept of a strongly regular graph (SRG) goes back to
the investigations of R.C.~Bose in relation to the design of statistical
experiments. For more than two decades it was considered in terms of
partially balanced incomplete block designs, while the term itself was
coined in~\cite{Bos63}.

There are a number of necessary conditions for the parameters $(v, k,
\lambda, \mu)$ of a putative SRG, which are formulated by means of
standard (or slightly more sophisticated) tools from spectral graph
theory (cf.~\cite{BroH12}) and which are called {\em feasibility
conditions.}

There are many known infinite classes and series of SRG's, in
particular those coming from finite permutation groups and finite
geometries.

There is a continuous interest to know all SRG's (up to isomorphism)
for a given feasible parameter set. A lot of such information is
accumulated on the home page of Andries Brouwer~\cite{Bro:www}.

On the other hand there exist so-called {\em prolific constructions}\/
of SRG's (in the sense of~\cite{CamS02}); this usually means that for
a given infinite series of feasible parameter sets the number of SRG's
grows exponentially. Techniques to present such prolific constructions
were suggested by W.D.~Wallis and developed quite essentially by
D.~Fon-Der-Flaass and M.~Muzychuk, see \cite{Muz07}.

The discovery of prolific constructions put a final dot in the
understanding of the fact that the full classification of SRG's is a
hopeless problem.

One who is interested in ``nice'' SRG's is forced to rigorously
formulate additional requirements to the considered objects, using suitable
group-theoretical or combinatorial language.

\subsection{$k$-isoregular graphs}
\label{sub:isoregular}

The concept of a $k$-isoregular graph has two independent origins,
both related to the investigation of rank~3 permutation groups. The
first origin is the Ph.D.~thesis of J.M.J.~Buczak (1980), fulfilled at
Oxford. The main results of this thesis were briefly mentioned in
Section~8 of \cite{CamL91}, while the text itself was not available to
M.~Klin and his colleagues for three decades.

The other origin is the paper \cite{GolK78}, where all absolutely
homogeneous graphs were classified and first steps were taken towards
the investigation of $k$-homo\-geneous graphs for $k\ge 2$; note that
2-homogeneous graphs are exactly rank~3 graphs.

It should be mentioned that the investigation of absolutely
homogeneous graphs was an attractive goal for many other
researchers. More or less at the same time the full classification was
suggested in publications by G.~Cherlin, A.D.~Gardiner, H.~Enomoto,
Ch.~Ronse, J.~Sheehan and others, see, e.g., the detailed
bibliography in the book \cite{Che98}. 

It was Ja.~Gol'fand who first realized that the concept of a
$k$-homogeneous graph may be approximated in purely combinatorial
terms bz $k$-isoregular graphs. In particular, absolutely regular graphs
coincide with the absolutely homogeneous graphs, and the proof of this
fact may be achieved without any use of group-theoretical
arguments. An outline of the proof (due to Gol'fand) is presented in
the Section~4.3 of \cite{KliPR88}.

In fact, the history of this proof is just a small visible part of the
related drama of ideas. Around 1979-80 Gol'fand prepared a manuscript
with the claim that each 5-regular graph is absolutely
regular. Unfortunately a fatal mistake (discovered by M.~Klin)
appeared on the very last page of this detailed text. However, all
previous results in this ingenious text were correct, thus the real
claim obtained by Gol'fand was the description of two putative
infinite classes of parameter sets for 4-regular graphs. One class
corresponds to absolutely regular graphs, while the other class,
denoted $M(r)$ by Gol'fand, was presenting an infinite series of
possible parameters. The graph $M(1)$ in this series corresponds to the
famous Schl{\"a}fli graph on 27 vertices, while $M(2)$ is the unique
McLaughlin graph with the parameters $(275, 112, 30, 56)$. The
existence of the graphs $M(r)$ for $r\ge 3$ remained open. (A brief
introduction about this series is also available in \cite{KliPR88}.)

Klin and his colleagues made a lot of efforts to convince Gol'fand to
publish his brilliant result (without the final page), however they
never succeeded.

Gol'fand spent many years in attempts to reach the full desired
result. After the collapse of the USSR he lived in relative
poverty. At the end of the century he was killed in his apartment
under unclear circumstances.

Around 1995, when he was still alive, Klin and Woldar established an
attempt to publish a series of papers, starting from Gol'fand's text as
Part I and with \cite{KliW} as Part II. Unfortunately after Gol'fand's
death the fate of his heritage still remains indefinite, this is one
of the reasons why the work over the unfinished text \cite{KliW} was
conserved. 

It should be mentioned that the original term $k$-regular suggested by
Gol'fand lead to possible confusion. Indeed, for many researchers in
graph theory, $k$-regular means regular of valency $k$. This is why
Klin and Woldar in \cite{KliW} decided to use the new term {\em
  $k$-isoregular}, which does not imply any confusion.

We refer to \cite{Rei00} for all necessary precise formulations
related to the concept of $k$-isoregularity.

One more approach, influenced by the consideration of graphs with the
$t$-vertex condition is developed in publications by J.~Wojdylo, see
\cite{Woj98} for example.

\subsection{Smith graphs}
\label{sub:smith}

In 1975 a few papers were published by M.~Smith, see for example
\cite{Smi75}, devoted to rank 3 permutation groups with the property
that each of the two transitive subconstituents is also a rank 3
group. All possible feasible parameters for the corresponding SRG's
were presented in evident form. At that time the result was regarded
as a program for further attacks toward a better understanding of such
graphs. A few years later, with the announcement of the classification
of finite simple groups, this research program became obsolete.

Since that time any SRG with the parameters described by Smith is
called a {\em Smith graph\/}, be it rank 3 or not. A great
significance of Smith graphs in the current context follows from the
the paper \cite{CamGS78}, in which those SRG's were investigated for
which for each vertex $x$ both induced subgraphs $\Gamma_1(x)$ and
$\Gamma_2(x)$ are also SRG's. It is easy to check that this class of
graphs strictly coincides with 3-isoregular graphs. (Note that the
complete graph and the empty graph may be regarded as degenerate
SRG's.)

The main result in \cite{CamGS78} claims that each 3-isoregular graph
is either the pentagon, a pseudo Latin square or negative Latin square
graph, or (up to complement) a Smith graph. Another significant
ingredient of \cite{CamGS78} is the established link of the class of
3-isoregular graphs with extremal properties of the Krein parameters
of SRG's and spherical $t$-designs (this link will be briefly
mentioned in Section~\ref{sub:interactions}). 

We should also mention Section 8 of \cite{CamGS78} which deals with
generalised quadrangles of order $(q,q^2)$. It was proved there that such
geometrical structures coexist with orthogonal arrays of strength $3$
and order $q$, as well as with certain codes over an alphabet with $q$
letters. This coexistence creates an additional challenge for the
reader to try to obtain our main result in absolutely different
context, relying only on the geometrical analysis of the related
codes.

We wish also to mention that 3-isoregular graphs are frequently called
triply regular graphs, which from time to time are subject of ongoing
research, see \cite{Guo11}. Hopefully this section of our paper will
help the modern investigators to better comprehend all the facets of
these striking combinatorial objects.

The concept of a 3-isoregular graph can be generalised to an arbitrary
symmetric association scheme. In this more general approach links with
spherical designs are well visible, see \cite{Sud10}.

\subsection{The $4$-vertex condition}
\label{sub:4-vertex}

Here we provide more details regarding the consideration of the
$4$-vertex condition.

The paper \cite{Iva94} (the last research input by Andrei Ivanov
before his transfer to the software industry) introduces two infinite
families of graphs on $2^{2n}$ vertices via the use of suitable finite
geometries. It also contains a rigorous proof of a folklore claim that
in the definition of the $t$-vertex condition for a regular graph it
is enough to check its fulfillment just on edges and non-edges (that
is, the inspection of loops is redundant).

Note that \cite{Iva94} contains interesting conjectures about the
induces subgraphs of his graphs which still remain unnoticed by modern
researchers (although one of the conjectures was confirmed in
\cite{Rei00}). 

Besides the well-known paper \cite{HesH71}, Higman himself was
considering graphs with the $4$-vertex condition in (at least) one
more paper \cite{Hig71}. In particular, he considered the block graphs
of Steiner triple systems with $m$ vertices (briefly $STS(m)$), which
are known to be SRG's. Highman proved that the block graph satisfies
the $4$-vertex condition if the $STS(m)$ consists of the points and
lines in projective space $PG(n,2)$ over the field of two elements,
here, $m=2^n-1$. Besides this, the $4$-condition may be satisfied for
$STS(m)$, for $m\in\{9,13,25\}$.

While the cases $m=9$ and $m=13$ can be easily settled, the case
$m=25$ remained open for four decades. Following the advice of
M.~Klin, P.~Kaski et al obtained the negative answer (see
\cite{KasKO12}), essentially relying on the use of computer and some
extra clever ad hoc tricks (the complete enumeration of all $STS(25)$
looks hopeless at the moment).

There are also known a few sporadic proper (non rank 3) SRG's which
satisfy the 4-condition. The smallest such graphs have 36 vertices, they
were carefully investigated in \cite{KliMRR05}. 

\subsection{Generalisations for association schemes}
\label{sub:generalisations}

Every SRG $\Gamma$ together with the complement $\overline \Gamma$
forms a symmetric association scheme with two classes (or of rank
3). Therefore it is natural to consider the concept of the $t$-vertex
condition for arbitrary association schemes.

Non-symmetric schemes with two classes are equivalent to pairs of
complementary doubly regular tournaments. Each such (skew-symmetric)
scheme satisfies the $4$-vertex condition \cite{Pas92}.

The situation for higher ranks is much more sophisticated and goes beyond
the scope of this survey. The first results in this direction were
presented in \cite{Fur87} and \cite{CaiFI92}, though without the
evident use of association schemes. Further steps stem from
\cite{EvdP99}. 

\subsection{Interactions with other problems}
\label{sub:interactions}

During the last two decades highly symmetrical assocaiation schemes,
and in particular SRG's, were used quite essentially in a number of
significant research approaches in diverse parts of mathematics. For
each such case it is not so easy to recognize immediately which
properties of the related graphs or schemes are exactly requested. The
reason of difficulties is as a rule a serious discrepancy in the
languages adopted in the corresponding parts of mathematics.

A few examples are briefly mentioned below, each together with at least one
striking reference.

Two decades ago F.~Jaeger demonstrated in \cite{Jae92} how one could
produce spin models for the applications in the theory of link
invariants (in the sense of V.R.F.~Jones), using suitable self-dual
association schemes. The requested SRG's should be 3-isoregular. 

The problem of description of finite point systems in Euclidean space
which satisfy some energy minimization criteria was attacked by
H. Cohn et al in a number of publications throughout the last
decade. Surprisingly, they established links of putative optimal
systems with some famous SRG's, among them there are a few
3-isoregular graphs, see \cite{CohK07}.

Deterministic polynomial factoring is considered in the framework of
$m$-schemes as they are called in \cite{IvaKS09}. In fact, this class
of objects is strictly related to diverse kinds of highly symmetrical
association schemes (as they were mentioned in
Section~\ref{sub:generalisations}). The recent thesis \cite{Aro13} is a
good source to digest sophisticated interactions, which appear in this
new line of research.

Spherical designs are intimately related to the existence of some
classes of SRG's, in particular to putative $3$- and $4$-isoregular
graphs. The survey \cite{CohK07} may be helpful as an initial
introduction. In this context, the negative results presented in
\cite{BanMV04} imply the non-existence of infinitely many 4-isoregular
graphs. This is, in a sense, a partial fulfilment of the foremost
dream of the late Ja.~Gol'fand. It should be mentioned that such an
implication is not immediately visible for a non-perplexed reader.

\section{Conclusion}
We have shown that the point graphs of generalised quadrangles of
order $(s,s^2)$ satisfy the 7-vertex condition. This provides us with
an infinite family of strongly regular graphs satisfying that
condition which have intransitive automorphism groups.

Hence, if Klin's Conjecture holds (i.e., if there is a number $t_0$ such
that the $t_0$-vertex condition implies a rank 3 automorphism group) 
then $t_0\ge 8$.

Still, the proof of this conjecture in its full generality appears
intractable. If we restrict ourselves to point graphs of generalised
quadrangles, then it looks reasonable to prove a similar statement, in
particular since many characterization of the ``classical''
generalised quadrangles (which include those with a rank 3 group) are
known. 

\bibliographystyle{plain}
\bibliography{general}
\end{document}